\documentclass[twoside,a4paper]{amsart}
\usepackage{amssymb, amsmath,latexsym, amsthm}
\usepackage{fancyhdr}

\newtheoremstyle{dotless}{}{}{\itshape}{}{\bfseries}{}{ }{}
\theoremstyle{dotless}
\newtheorem{theorem}{Theorem}
\newtheorem*{theoremm}{Theorem 2$^\prime$}
\newtheorem{theoremmm}{Theorem}
\newtheorem{lemma}{Lemma}
\theoremstyle{definition}

\DeclareMathOperator{\dist}{dist}

\title{On a theorem of M. Cartwright in higher dimensions}
\author{A. Logunov}
\address{Chebyshev Laboratory at Saint-Petersburg State University, 198904, Saint-Petersburg, Russia}
\email{log239@yandex.ru}
\author{E. Malinnikova}
\address{Department of Mathematical Sciences, Norwegian University of Science and Technology, NO-7491, Trondheim, Norway}
\email{eugenia.malinnikova@math.ntnu.no}
\author{P. Mozolyako}
\address{Department of Mathematical Sciences, Norwegian University of Science andechnology, NO-7491, Trondheim, Norway}
\email{pavel.mozolyako@math.ntnu.no}
\thanks{
A.L. was supported by the Chebyshev Laboratory (Department of Mathematics and Mechanics, St. Petersburg State University) under RF government grant 11.G34.31.0026 and JSC "Gazprom Neft". E.M. was supported by Research Council of Norway, grant 213638. P.M. was supported by Research Council of Norway, grant 204726/V30.}
\begin{document}

\begin{abstract}
We consider harmonic functions in the unit ball of $\mathbb{R}^{n+1}$ that are unbounded near the boundary but can be estimated from above by some (rapidly increasing) radial weight $w$. Our main result gives some conditions on $w$ that guarantee the estimate from below  on the harmonic function by a multiple of this weight. In dimension two this reverse estimate was first obtained by  M.~Cartwright for the case of the power weights, $w_p(z)=(1-|z|)^{-p},\ p>1$, and then generalized to a wide class of regular weights by a number of authors.  
\end{abstract}
\maketitle
\renewcommand{\thetheoremmm}{\Alph{theoremmm}}
\normalsize
\section{Introduction}
A well-known theorem by M. Cartwright \cite{Car:33} states that if a function $u$, harmonic in the unit disk, $u(0) = 0$, satisfies the one-sided growth condition
\begin{equation}\notag
u(z) \leq w(1-|z|),\quad z\in \mathbb{D},
\end{equation}
where $w(t) = \frac{1}{t^p}$ for some $p>1$, then the reverse inequality holds
\begin{equation}\notag
u(z) \geq -Cw(1-|z|),\quad z\in \mathbb{D},
\end{equation}
where $C$ depends only on $p$. This result was later refined (\cite{Car:35,L:56,L:62}) and extended by M.~ Cartwright herself and C.N.~Linden to more general weights.
The works by  N.~Nikolskii (\cite{Nik}) and A.~Borichev (\cite{Bor:03}) should also be mentioned, the latter in particular, where a very nice estimate $u(z) \geq -(1+o(1))w(1-|z|)$ was obtained for sufficiently fast growing weights (see \cite[section 1.3]{Bor:03});  some estimates for the constant in the reverse inequality were also given  earlier in \cite{L:62}. The techniques used in all of the works mentioned above involve analytic functions and conformal mappings and are therefore limited to the complex plane. However, it is natural to ask if similar results hold for harmonic functions in higher dimensions, related problems in higher dimensions were studied by P.J. Rippon, K. Samotij and B. Korenblum, see \cite{R:82, S:86, S:87, KRS}. In this paper we extend the theorem of Cartwright to harmonic functions in the unit ball in $\mathbb{R}^{n+1}$. The result holds for a large class of (regular) weights.

Let $w:\mathbb{R}_+\mapsto\mathbb{R}_+$ be a strictly decreasing function, such that
$\lim_{y\rightarrow 0}w(y)=\infty$ and $w(1) =1$.
Furthermore we assume that $w\in C^2$ and the following growth and regularity conditions are satisfied
\begin{equation}\label{e:als}
\lim_{y\rightarrow0}\frac{w(y)}{w'(y)} = 0,
\end{equation}
and
\begin{equation}\label{e:ar}
\left(\frac{w(y)}{w'(y)}\right)'\geq -\frac{1-\delta}{n}, \quad 0< y < 1,
\end{equation}
for some positive $\delta$.
Our main result is the following
\begin{theorem}\label{t:th1}

Let $U$ be a harmonic function in the unit ball $B\subset \mathbb{R}^{n+1}$, $U(0)=0$. Assume that $U$ admits the growth condition
\begin{equation}\label{e:ugc}
U(z) \leq w(1-|z|),\quad z\in B,
\end{equation}
where the weight $w$ satisfies \eqref{e:als} and \eqref{e:ar} above.
 Then the following two-sided estimate holds
\begin{equation}\label{e:utsgc}
\left|U(z)\right| \leq Cw(1-|z|),\quad z \in B,
\end{equation}
where the constant $C = C(n,\delta)$ depends only on the parameter $\delta$ and dimension $n$.
\end{theorem}

The conditions \eqref{e:als} and \eqref{e:ar} assure that the weight $w$ grows relatively fast as $|z|\rightarrow 1$ and is regular. The natural regularity for majorants of harmonic functions is logarithmic convexity, however it is shown in \cite[Proposition 4.1]{Bor:03} that some additional regularity of the weight is necessary for Theorem \ref{t:th1} to hold.\par For the rate of the growth, the weight $w_0(y)=y^{-n}$ is the natural threshold in this result. We see that $w(y)=y^{-p}$ satisfies \eqref{e:als} and \eqref{e:ar} if and only if $p>n$. The result also fails when $p\le n$, one can easily see that the Poisson kernel for the ball $B$ is strictly positive, but grows like $w_0$ near its singularity at the boundary. There is no upper bound on the growth of $w$ (for example $w = \exp(\frac{1}{y})$ or $w = \exp(\exp(\exp(\frac{1}{y})))$ fit well into \eqref{e:als} and \eqref{e:ar} for small enough $y$, see section \ref{ss:sc} for further discussion) as long as the weight $w$ is \textquotedblleft regular\textquotedblright. The nature of conditions \eqref{e:als} and \eqref{e:ar}  is further discussed in section \ref{ss:sc}.\par
For the case of the weight $w_0$ we have the following generalization of the two-dimensional result of Cartwright.
\begin{theorem}\label{th:N3}
Let $U$ be a harmonic function in the unit ball $B\subset \mathbb{R}^{n+1}$,  $U(0)=0$. Assume that
\begin{equation}\notag
U(z)\leq (1-|z|)^{-n},\quad z\in B.
\end{equation}
Then
\begin{equation}\notag
|U(z)|\leq C(1-|z|)^{-n}\left(\log\frac1{1-|z|}\right)^{n+1},\quad |z|>1/2
\end{equation}
where $C$ depends on $n$ only.
\end{theorem}
We also show that the reverse estimate in this theorem is the best possible. Further results for other weights can be obtained by methods developed here, for example some logarithmic factors can be added to the weight $w_0$ as in Theorem 2 of \cite{L:56}.\par
For weights that grow slower than $w_0$ we recover the following result of Rippon \cite{R:82}.
\begin{theoremm}[Rippon]
Let $U$ be a harmonic function in the unit ball $B\subset \mathbb{R}^{n+1}$, $U(0)=0$. Assume that 
\begin{equation}\notag
U(z)\leq w(1-|z|), \quad z\in B,
\end{equation}
where $w$ is strictly decreasing, absolutely continuous function such that
\begin{equation}\label{cond:int}
I_0=\int_0^1\left(\frac{w(t)}{t}\right)^{\frac1{n+1}}dt < +\infty.
\end{equation}
Then 
\begin{equation}\notag
|U(z)|\le C(1-|z|)^{-n}, \quad  |z|>1/2, 
\end{equation}
where $C$ depends on $n$ and $I_0$.
\end{theoremm} 
The last statement contains the third result of Cartwright (for $n=1$ and $w(y) = y^{-a},\; a<1$) and Theorem 1 of Linden \cite{L:56} (which corresponds to the case $n=1$, weight $w$ is regular and satisfies the integral condition above). In dimension two ($n=1$) the statement also follows from results of W. Hayman and B. Korenblum \cite{HK:76}. We suggest a unified approach that allows us to prove Theorems \ref{t:th1}, \ref{th:N3} and 2$^\prime$ more or less simultaneously.\par
The plan of the paper is as follows. In section \ref{s:not} we collect some notation and technical results. The proof of Theorems \ref{t:th1}-\ref{th:N3}$'$ consists of two mostly independent parts: first, through the use of Harnack inequality, we reduce the estimate \eqref{e:utsgc} to the two-sided inequality for some averages of the function and reformulate the regularity conditions \eqref{e:als} and \eqref{e:ar} into more convenient (with relation to our context) form, this is covered in sections  \ref{ss:th2.5} and \ref{ss:tl} respectively. The second part of the proof is given in the section \ref{ss:stth.1}, where we obtain the required estimates for the averages of the harmonic function over some spherical caps. The argument involves a construction of an auxiliary surface, this construction is a slight modification of one used in \cite{KRS}, where the Poisson representation of harmonic functions satisfying one-sided growth restriction was discussed. A similar construction appeared earlier in \cite{R:82} for the proof of Theorem 2$^\prime$. In the last section  we compare our regularity conditions with those given in \cite{Bor:03} and  construct an example demonstrating that the estimate in Theorem \ref{th:N3} is the best possible.

\section{Two-sided inequality for the averaged function}\label{s:aux}
\subsection{Notation}\label{s:not}
Given two functions $f$ and $g$ we say that $f\lesssim g$ if there is a positive constant $C$, depending only on the dimension $n$, such that $f\leq Cg$. We write $f\sim g$ if $f\lesssim g$ and $g\lesssim f$ simultaneously. A point $z$ in the unit ball $B$ will be denoted by $(x,y)$, where $x \in S= \partial B,\, x= \frac{z}{|z|}$ and $y = 1-|z|>0$. Then $y$ is the distance from $z$ to the unit sphere and $x$ is the closest to $z$ point on the sphere, this notation turns out to be convenient for our problem.   Despite the inconsistency we will sometimes write $u(z)$ and sometimes $u(x,y)$.
By $P_{y}(x,\xi)$ we denote the Poisson kernel for $B$
\begin{equation}\notag
P_{y}(x,\xi) = \frac{y(2-y)}{|(1-y)x - \xi|^{n+1}} = \frac{1-|z|^2}{|z-\xi|^{n+1}},\quad x,\xi \in S,\; y\in [0,1], \; z = (1-y)\cdot x.
\end{equation}
Let also $\phi(z,\zeta)\in [0,\pi]$ be the angle between  $z$ and $\zeta$,\; $z,\zeta\in\mathbb{R}^n\setminus\{0\}$,
\begin{equation}\notag
\phi(z,\zeta) = \cos^{-1}\left(\frac{\langle z,\zeta\rangle}{|z||\zeta|}\right).
\end{equation}
Let $\eta$ be the south pole of $B$, $\eta = (0,\dots,0,-1)$, we fix this notation for the rest of the paper.. Given $t\leq \pi$ we denote by $A(y,t)$ the \textquotedblleft antarctic\textquotedblright\, cap
\begin{equation}\notag
A(y,t) = \{z\in B: \, |z| = 1-y,\, \phi(z,\eta) \leq t\},
\end{equation}
also we put $S(y,t) = \partial A(y,t)$.
Consider then the averaged Poisson kernel,
\begin{equation}\notag
\mu(x,y,t) = \frac{1}{\sigma_{n-1}(S(0,t))}\int_{S(0,t)}P_y(x,\xi)\,d\sigma_{n-1}(\xi),\quad x\in S,\; 0<y\leq 1,\; 0\leq t\leq \pi,
\end{equation}
where $\sigma_{n-1}$ is the $(n-1)$-dimensional surface measure on $S(0,t)$, $\sigma_{n-1}(S(0,t)) = C(n)\sin^{n-1}t$. Note that $\mu(x,1,t) = 1,$ for $ x\in S,\, 0\leq t\leq \pi$.

We need the following estimate (Lemma \textbf{1} from \cite{KRS})
\begin{lemma}\label{l:pki}
For any $x\in S$ and any $y\in (0,1],\; t\leq \phi(x,\eta)$ we have
\begin{equation}\notag 
\mu(x,y,t) \sim \frac{y}{d^2(d^{n-1}+\sin^{n-1}\phi(x,\eta))},
\end{equation}
where $d^2 = 1+(1-y)^2 -2(1-y)\cos(\phi(x,\eta)-t) = \dist^2(S(0,t), S(y,\phi(x,\eta)))$.
\end{lemma}
This averaged Poisson kernel will be useful later on (section \ref{ss:stth}), when we deal with the axially symmetric harmonic functions. We call the function $\tilde{u}$ on the unit ball axially symmetric if $u(z)$ depends only on $|z|$ and the angle $\phi(z,\eta)$ between the argument $z$ and $\eta = (0,\dots,0,-1)$. If such a function has boundary values $\tilde{u}(x,0) = \varphi(t),\; t = \phi(x,\eta)$, we rewrite the Poisson representation formula in the following way
\begin{multline}\label{e:sprf}
\tilde{u}(x,y) = \int_{S}\tilde{u}(\xi,0)P_{y}(x,\xi)\,d\sigma_n(\xi) = \int_{0}^{\pi}\int_{S(0,t)}\tilde{u}(\xi,0)P_y(x,\xi)\,d\sigma_{n-1}(\xi)\,dt\\
 = C(n)\int_{0}^{\pi}\varphi(t)\mu(x,y,t)\sin^{n-1}t\,dt,
\end{multline}
where $\sigma_n$ is the normalized surface measure on $S$ and $C(n)$ is the surface measure of the $(n-1)$-dimensional unit sphere.
\subsection{Averaging theorem}\label{ss:th2.5}
In this section we show that in order to prove \eqref{e:utsgc} it is sufficient to know the estimate of the 
averages of $U$ over some spherical caps $A(\theta, \alpha)$ for some $\alpha=\alpha(\theta)$. We also make some preliminary estimates to deduce inequalities for the averages from \eqref{e:ugc} and the regularity conditions \eqref{e:als} and \eqref{e:ar}.

First we prove the following theorem
\begin{theorem}\label{t:th2.5}
Let $U$ be a harmonic function in $B$, continuous up to the boundary, satisfying
\begin{equation}\notag
\begin{split}
&U(0) = 0;\\
&U(x,y) \leq w(y),\quad x\in S,\; 0\leq y \leq 1,
\end{split}
\end{equation}
where $w$ is a strictly decreasing function.
Assume that for some positive $\theta < \frac{1}{2}$ there exists a positive $\alpha = \alpha(\theta) \leq \frac{\theta}{4}$ such that
\begin{subequations}
\begin{eqnarray}
\label{e:wav.1}&w(\theta-2\alpha) \leq C_1w(\theta);\\
\label{e:wav.2}&\frac{1}{\alpha^n}\left|\int_{A(\theta,\alpha)}U(z)\,d\sigma_n(z)\right| \leq C_2w(\theta)
\end{eqnarray}
\end{subequations}
for some positive constants $C_1, \, C_2$.
Then 
\begin{equation}\label{e:leta}
U(\eta,\theta)\geq - C_3w(\theta),
\end{equation}
where $C_3 = C_3(C_1,C_2,n)$.
\end{theorem}
\begin{proof}
Consider the ball $b$ with center $(\eta,\theta)$ and radius $2\alpha$. The condition \eqref{e:wav.1} implies that for any $z\in B$ we have
\begin{equation}\notag
w(1-|z|) \leq C_1w(\theta),
\end{equation}
and therefore 
\begin{equation}\notag
-U(z) + C_1w(\theta) \geq 0, \quad z \in b.
\end{equation}
Now we can use the Harnack inequality to obtain
\begin{equation}\notag
-U(\eta,\theta) + C_1w(\theta) \leq C(n)\left(-U(z) + C_1w(\theta)\right),\quad |z| = 1-\theta,\;\phi(z,\eta) \leq \alpha.
\end{equation}
All that remains is to take average over $\{z:|z|=1-\theta,\; \phi(z,\eta) \leq\alpha\}$, so we get
\begin{equation}\notag
-U(\eta,\theta) + C_1w(\theta) \leq \tilde{C}(n)\frac{1}{\alpha^n}\int_{\{z:|z|=1-\theta,\; \phi(z,\eta) \leq\alpha\}}\left(-U(z) + C_1w(\theta)\right)\,d\sigma_n(z),
\end{equation}
which, combined with \eqref{e:wav.2}, implies \eqref{e:leta}.
\end{proof}
\subsection{Two lemmas}\label{ss:tl}
Now we want to show that the regularity conditions \eqref{e:als} and \eqref{e:ar} imply \eqref{e:wav.1} and \eqref{e:wav.2} for an appropriately chosen $\alpha = \alpha(\theta)$. It turns out that a natural way to define $\alpha$ (at least for somewhat smooth weights) is this one 
\begin{equation}\label{e:adf}
\alpha(\theta) := -\frac{w(\theta)}{10w'(\theta)},\quad 0<\theta<1,
\end{equation}
we refer the reader to Section \ref{ss:sc} for a further discussion. The validity of our choice is provided by the following lemma.
\begin{lemma}\label{l:lh1}
If the weight $w$ satisfies \eqref{e:ar} then $\alpha(\theta) \leq \frac{\theta}{4}$ and
\begin{equation}\notag
w(\theta-2\alpha(\theta)) \leq 2w(\theta),\quad 0< \theta \leq \frac{1}{2}.
\end{equation}
\end{lemma}
Now we need to see if the $\alpha$ we have chosen in \eqref{e:adf} fits into \eqref{e:wav.2}. This is a much more complicated task than verifying \eqref{e:wav.1}, and the first step is the statement below.
\begin{lemma}\label{l:lh2}
If the weight $w$ satisfies \eqref{e:als} and \eqref{e:ar} for some $\delta>0$ and $\alpha(\theta)$ is defined by \eqref{e:adf} then for $0<\theta\leq \frac{1}{2}$
\begin{equation}\label{e:wri}
\int_{0}^1\left(\frac{w(y(1-\theta)+\theta)}{y}\right)^{\frac{1}{n+1}}\,dy \leq \left(\frac{n+1}{n} + \frac{40(n+1)}{\delta}\right)w^{\frac{1}{n+1}}(\theta)\alpha^{\frac{n}{n+1}}(\theta).
\end{equation}
\end{lemma}
\subsection{Proof of Lemma \ref{l:lh1}}\label{ss:lh1}
For any $0<\theta<\frac12$ there exists  $\theta_1\in [\theta-2\alpha(\theta),\theta]$ such that
\begin{equation}\label{e:tht1}
w(\theta) = w(\theta-2\alpha(\theta)) + 2\alpha(\theta)w'(\theta_1).
\end{equation}
The regularity condition \eqref{e:ar} implies that 
\begin{equation}\label{e:smdtmp}
\alpha'(\theta) \leq \frac{1-\delta}{10n}.
\end{equation}
Hence $\alpha(\theta_1) \geq \alpha(\theta) - \frac{\theta-\theta_1}{10}$,
and, on the other hand, $\theta-\theta_1 \leq 2\alpha(\theta)$, therefore 
\begin{equation}\notag
\alpha(\theta_1) \geq \alpha(\theta) - \frac{\alpha(\theta)}{5} = \frac{4}{5}\alpha(\theta).
\end{equation}
We see that
\begin{equation}\notag
-\alpha(\theta)w'(\theta_1) = \frac{\alpha(\theta)w(\theta_1)}{10\alpha(\theta_1)} \leq \frac{w(\theta_1)}{8}.
\end{equation}
Plugging this inequality into \eqref{e:tht1}, we obtain
\begin{equation}\notag
w(\theta) \geq w(\theta-2\alpha(\theta)) - \frac{w(\theta_1)}{4} \geq \frac{w(\theta-2\alpha(\theta))}{2}
\end{equation}
and the lemma follows.

\subsection{Proof of Lemma \ref{l:lh2}}\label{ss:lh2}
We split the integral in \eqref{e:wri} into two parts
\begin{multline*}
\int_{0}^1\left(\frac{w(y+\theta-y\theta)}{y}\right)^{\frac{1}{n+1}}\,dy =\\ \int_{0}^{\alpha(\theta)}\left(\frac{w(y+\theta-y\theta)}{y}\right)^{\frac{1}{n+1}}\,dy + \int_{\alpha(\theta)}^1\left(\frac{w(y+\theta-y\theta)}{y}\right)^{\frac{1}{n+1}}\,dy.
\end{multline*}
To estimate the first integral we just note that  for $\theta \leq 1$
\begin{equation}\notag
\int_{0}^{\alpha}\left(\frac{w(y+\theta-y\theta)}{y}\right)^{\frac{1}{n+1}}\,dy \leq w^{\frac{1}{n+1}}(\theta)\int_{0}^{\alpha}\left(\frac{1}{y}\right)^{\frac{1}{n+1}}\,dy \leq \frac{n+1}{n}w^{\frac{1}{n+1}}(\theta)\alpha^{\frac{n}{n+1}}(\theta).
\end{equation}
To deal with the second one we let
$\kappa(y) := w^{\frac{1}{n+1}}(y),\ y>0,$
so we need to verify
that
\begin{equation}\label{e:kpi}
\int_\alpha^1\kappa((1-\theta)y+\theta)y^{-\frac{1}{n+1}}\,dy \leq \tilde{C}\kappa(\theta)\alpha^{\frac{n}{n+1}}(\theta).
\end{equation}
It follows from the definition of $\alpha$ and $\kappa$ that
\begin{equation}\notag
\alpha\cdot\frac{\kappa'}{\kappa} = -\frac{1}{10(n+1)},
\end{equation}
and, using \eqref{e:ar}, we obtain
\begin{multline}\label{e:akpd}
\left(\alpha^{\frac{n}{n+1}}\kappa\right)' = \alpha^{-\frac{1}{n+1}}\kappa\left(\frac{\kappa'}{\kappa}\alpha + \frac{n}{n+1}\alpha'\right) =\\\alpha^{-\frac{1}{n+1}}\kappa\left(-\frac{1}{10(n+1)} + \frac{n}{n+1}\alpha'\right)\leq
 -\frac{1}{10}\alpha^{-\frac{1}{n+1}}\kappa\left(\frac{1}{(n+1)} - \frac{1-\delta}{n+1}\right) =\\ -\frac{\delta}{10(n+1)}\alpha^{-\frac{1}{n+1}}\kappa. 
\end{multline}
Let $c= \frac{\delta}{10(n+1)}$, then
integrating \eqref{e:akpd} over $[\theta,1]$ for $\theta<\frac12$ we see that
\begin{multline*}
\kappa(\theta)\alpha^{\frac{n}{n+1}}(\theta) \geq c\int_{\theta}^{1}\kappa(y)\alpha^{-\frac{1}{n+1}}(y)\,dy  =\\
c(1-\theta)\int_0^1\kappa\left(y(1-\theta)+\theta\right)\alpha^{-\frac{1}{n+1}}\left(y(1-\theta)+\theta\right)\,dy 
\end{multline*}
Now, for $y\geq\alpha(\theta)$ by \eqref{e:smdtmp}
we have $\alpha(\theta + (1-\theta)y) \leq \alpha(\theta)+(1-\theta)y \leq 2y$. Therefore
\begin{equation}\notag
\int_0^1\kappa\left(y(1-\theta)+\theta\right)\alpha^{-\frac{1}{n+1}}\left(y(1-\theta)+\theta\right)\,dy  \geq 2^{-\frac{1}{n+1}}\int_{\alpha(\theta)}^1\kappa\left((1-\theta)y+\theta\right)y^{-\frac{1}{n+1}}\,dy. 
\end{equation}
This gives  \eqref{e:kpi} with $\tilde{C} = \frac{2^{\frac{1}{n+1}}10(n+1)}{\delta(1-\theta)}$.
Summing up the estimates for both integrals, we get
\begin{equation}\notag
\int_{0}^1\left(\frac{w(y+\theta-y\theta)}{y}\right)^{\frac{1}{n+1}}\,dy \leq \left(\frac{n+1}{n} + \frac{40(n+1)}{\delta}\right)w^{\frac{1}{n+1}}(\theta)\alpha^{\frac{n}{n+1}}(\theta),
\end{equation}
and we are done.
\subsection{Some comments}\label{ss:sc}
We have seen that in order to prove the main theorem we need the conditions \eqref{e:wav.1} and \eqref{e:wav.2}. They are rather independent: the proof of the first one is self-contained, and the second one, as it will be shown later, follows from Lemma \ref{l:lh2}, where we do not use any information on the doubling property of $\alpha$.
Combining them, we see  that for every fixed $\theta$ we essentially need to find some $\alpha = \alpha(\theta)$ such that
\begin{subequations}\label{e:rar}
\begin{eqnarray}
\label{e:rar.1}&w(\theta-\alpha) \leq 2w(\theta);\\
\label{e:rar.2}&\int_{\alpha}^1\left(\frac{w(y+\theta-y\theta)}{y}\right)^{\frac{1}{n+1}}\,dy \leq C(w,n)w^{\frac{1}{n+1}}(\theta)\alpha^{\frac{n}{n+1}}(\theta).
\end{eqnarray}
\end{subequations}
These two inequalities are actually \textquotedblleft fighting\textquotedblright$\,$ with each other. Indeed, if we put $\alpha$ to be very small, then the first condition is immediately satisfied, but the second one fails miserably. On the other hand $\alpha$ can not be large (compared to $\theta$), because of the first condition: the faster the weight $w$ grows the smaller must $\alpha$ be. If we try to unify these two inequalities we (albeit probably with some loss of information) will arrive to the \textquotedblleft regularity\textquotedblright\, of the weight $w$ as stated in \eqref{e:als} and \eqref{e:ar}.

It should be noted that this is the only place we needed the regularity conditions, so if we have a weight $w$ that satisfies \eqref{e:rar.1} and \eqref{e:rar.2} with some $\alpha$ (not necessarily defined like in \eqref{e:adf}), then Theorem \ref{t:th1} still holds. One important example (which will be discussed again in section \ref{s:bc}) is the weight $w$ of polynomial growth. Assume that $w\in C^1$ and 
\begin{equation}\label{e:wpg}
-\frac{N}{y} \leq \frac{w'(y)}{w(y)} \leq -\frac{n+\varepsilon}{y},\; y\in (0,1]
\end{equation}
for some positive $\varepsilon$ and $N\geq n+\varepsilon$. Put $\alpha(\theta) = \frac{\theta}{2N}$. Clearly $w(\theta - \alpha) \leq 2w(\theta),\; \theta\in (0,1]$, so that we have \eqref{e:rar.1}. Furthermore 
\begin{equation}\notag
\left(y^{\frac{n}{n+1}}w^{\frac{1}{n+1}}(y)\right)' = \frac{y^{-\frac{1}{n+1}}w^{\frac{1}{n+1}}(y)}{n+1}\left(\frac{w'(y)}{w(y)}y+n\right)
\end{equation}
so that
\begin{equation}\notag
\left(\alpha^{\frac{n}{n+1}}(y)w^{\frac{1}{n+1}}(y)\right)' \leq -C\alpha^{-\frac{1}{n+1}}(y)w^{\frac{1}{n+1}}(y),
\end{equation}
which is basically \eqref{e:akpd}. Following by letter the proof of the Lemma \ref{l:lh2} we see that \eqref{e:rar.2} also holds. Note that in this case the weight $w$ can be a little less smooth than required by the regularity condition \eqref{e:ar}.

Note also that in order to bound $U$ from below we do not need $w$ to be regular on the entire interval $(0,1]$. Assume that $w\in C^2$, decreasing, and \eqref{e:als} and \eqref{e:ar} hold only for $0< y \leq y_0$ (or $w\in C^1$ and \eqref{e:wpg} holds only for $0< y \leq y_0$) for some $y_0<1$. We still can prove a version of Theorem \ref{t:th1} replacing \eqref{e:utsgc} with
\begin{equation}\label{e:wnr}
|U(z)| \leq C_1 + C_2w(1-|z|), \quad z\in B.
\end{equation}
Indeed, it is easy to show that there exists a $C^{2}$ function $\tilde{w}$ that satisfies \eqref{e:als} and \eqref{e:ar} for $y\in (0,1]$ and such that
\begin{equation}\notag
\begin{split}
&\tilde{w}(y) \geq w(y),\quad y_1\leq y\leq 1;\\
&\tilde{w}(y) = Aw(y),\quad 0< y\leq y_1. 
\end{split}
\end{equation}
For example one may choose $\tilde{w}(y) = c(y+b)^{s}$ for $y\geq y_1$ and some $y_1\leq y_0$ such that $\left(\frac{w}{w'}\right)'(y_1) < 0$. Since Theorem \ref{t:th1} holds for $\tilde{w}$, we immediately have \eqref{e:wnr}. Similar argument works for $w\in C^1$ satisfying \eqref{e:wpg}.
\subsection{A question on harmonic measure estimates in cusp-like domains.}
One of the possible ways to simplify the proofs of Theorems \ref{t:th1}, \ref{th:N3} and \ref{th:N3}$^\prime$ is to obtain lower estimates for the asymptotic of the harmonic measure of regular axially symmetric cusp-like domains. The form of the domain  depends on $w$.  We refer the reader to \cite{R:07} where upper (not lower) estimates for harmonic measure in a cusp-like domains are used in connection  to the Levinson "loglog" theorem. Unfortunately, we do not know any reference for the lower estimates of harmonic measure in cusp-like domains in higher dimensions. We were  compelled to use the ideas developed in \cite{KRS} as a workaround to avoid the harmonic measure estimates.

\section{Main technical theorem}\label{ss:stth.1}
\subsection{Statement}\label{ss:stth}
The next theorem allows us to estimate the absolute values of some averages of the harmonic function bounded from above.
\begin{theorem}\label{l:lh3}
Let $\tilde{k}:\mathbb{R}_+\mapsto\mathbb{R}_+$ be a strictly decreasing absolutely continuous function such that
\begin{subequations}\label{e:kav}
\begin{eqnarray}
\label{e:kav.1}&\tilde{k}(0) < \infty;\\
\label{e:kav.2}&\int_{0}^{1}\left(\frac{\tilde{k}(y)}{y}\right)^{\frac{1}{n+1}}\,dy \leq D 
\end{eqnarray}
\end{subequations}
for some constant $1< D < \infty$.
Let $\tilde{u}$ be a harmonic function in $B$, continuous up to the boundary, satisfying
$\tilde{u}(0) = 0$ and
$\tilde{u}(x,y) \leq \tilde{k}(y)$ for $x\in S,\; 0\leq y \leq 1.$
Then for any $x_0\in S$ and $\beta\in[0,\frac{1}{2}]$ the following inequality holds
\begin{equation}\label{e:tsae}
\int_{\{\phi(x,x_0)\leq \beta\}}\tilde{u}(x,0)\,d\sigma_n(x)  \geq - C\left(D ^{n+1}+\tilde{k}(0)\beta^{n}\right). 
\end{equation}
where the constant $C$ depends only on the dimension $n$.
\end{theorem}
\subsection{Theorems \ref{l:lh3} and \ref{t:th2.5} imply Theorem \ref{t:th1}}\label{ss:dth1}
 Fix any positive $\theta\leq \frac{1}{2}$. Let $U$ and $w$ be as in Theorem \ref{t:th1}, and $\alpha$ be defined as in \eqref{e:adf}. The weight we are going to plug into Theorem \ref{l:lh3} is defined as follows
\begin{equation}\notag
\tilde{k}(y) := \frac{w(y+\theta-y\theta)}{w(\theta)\alpha(\theta)^n},\quad 0\leq y\leq 1.
\end{equation}
Indeed, if we apply Lemma \ref{l:lh2}, we obtain
\begin{equation}\notag
\int_{0}^{1}\left(\frac{\tilde{k}(y)}{y}\right)^{\frac{1}{n+1}}\,dy = \int_{0}^{1}\left(\frac{w(y+\theta-y\theta)}{w(\theta)\alpha^n(\theta)y}\right)^{\frac{1}{n+1}}\,dy \leq \left(\frac{n+1}{n} + \frac{40(n+1)}{\delta}\right),
\end{equation}
so we have the condition \eqref{e:kav.2} with $D = \frac{n+1}{n} + \frac{40(n+1)}{\delta}$. Now put $\beta = \alpha(\theta)$. Clearly, $\tilde{k}(0) <\infty$, and 
the function
\begin{equation}\notag
\tilde{u}(z) := \frac{U(z\cdot(1-\theta))}{w(\theta)\alpha^n(\theta)}, \quad |z|\leq 1,
\end{equation}
can be estimated from above
$\tilde{u}(x,y) \leq \tilde{k}(y).$
Theorem \ref{l:lh3} will therefore imply 
\begin{equation}\notag
\int_{\{\phi(x,x_0)\leq \alpha(\theta)\}}\tilde{u}(x,0)\,dx  \geq -C(n)\left(D ^{n+1}+\tilde{k}(0)\alpha^{n}(\theta)\right) \geq -C(n)\left(D ^{n+1}+1\right)
\end{equation}
which means that
\begin{equation}\notag
\frac{1}{\alpha^n(\theta)}\left|\int_{A(\theta,\alpha(\theta))}U(z)\,d\sigma_n(z)\right| \lesssim D^{n+1}w(\theta),
\end{equation}
and we get \eqref{e:wav.2}. The condition \eqref{e:wav.1} will follow from Lemma \ref{l:lh1}.
We arrive to the following inequality
\begin{equation}\notag
\frac{1}{\alpha^n(\theta)}\left|\int_{A(\theta,\alpha(\theta))}U(z)\,d\sigma_n(z)\right| \lesssim \left(\frac{n+1}{n} + \frac{40(n+1)}{\delta}\right)^{n+1}w(\theta),
\end{equation}
which combined with Theorem \ref{t:th2.5} proves Theorem \ref{t:th1}.
\subsection{Theorems \ref{l:lh3} and \ref{t:th2.5} imply Theorems \ref{th:N3}$'$ and \ref{th:N3}}\label{ss:thms23}
Let now $U$ and $w$ be as in Theorem \ref{th:N3}$^\prime$. We fix some $\theta\le \frac12$ and apply Theorem \ref{l:lh3} to $\tilde{u}(z)=U(z(1-\theta))$ and $
\tilde{k}(y)=w(y+\theta-y\theta).$
Then (\ref{cond:int}) implies
\[
\int_0^1\left(\frac{\tilde{k}(y)}{y}\right)^{\frac{1}{n+1}}dy\le \int_{0}^1\left(\frac{w(y(1-\theta))}{y}\right)^{\frac1{n+1}}dy\le 2^{\frac{n}{n+1}}I_0.\]
We obtain
\[
\int_{\{\phi(x,x_0)\le\theta/4\}}\tilde{u}(x,0)dx\ge -C(2^nI_0^{n+1}+w(\theta)4^{-n}\theta^n)\ge -C_1 I_0^{n+1},\]
where $C_1=C_1(n)$. Clearly, $w(y)\le I_0^{n+1}y^{-n}$ since $w$ is decreasing and (\ref{cond:int}) holds.  Hence
\[
\left|\frac{4^n}{\theta^n}\int_{A(\theta,\theta/4)} U(z)d\sigma_n(z)\right|\lesssim I_0^{n+1}\theta^{-n}.\]
Now we apply Theorem \ref{t:th2.5} with the weight $I_0^{n+1}y^{-n}$.  We also take $\alpha(\theta)=\theta/4$. Then \eqref{e:wav.1} and \eqref{e:wav.2} hold with $C_1$ and $C_2$ that depend on $n$ only. Theorem \ref{th:N3}$'$ follows. 

If $U$ is a harmonic function that satisfies the conditions of Theorem \ref{th:N3}, we define 
$\tilde{u}(z)=U(z(1-\theta))$ and  $\tilde{k}(y)=(\theta+y(1-\theta))^{-n}.$
Then $\tilde{k}(0)=\theta^{-n}$  and
\begin{multline*}
\int_0^1\left(\frac{\tilde{k}(y)}{y}\right)^{\frac1{n+1}}dy\le
\int_0^\theta \theta^{-\frac{n}{n+1}}y^{-\frac1{n+1}}dy+\int_\theta^1(1-\theta)^{-\frac{n}{n+1}}y^{-1}dy\\ \le\frac{n}{n+1}+(1-\theta)^{-\frac{n}{n+1}}|\log \theta|.
\end{multline*} 
Applying Theorem \ref{l:lh3}, we get for $\theta\le 1/2$
\[
\left|\frac{4^n}{\theta^n}\int_{A(\theta,\theta/4)}U(z)d\sigma_n(z)\right|\le C(n)\left(1+\left(\log\frac{1}{\theta}\right)^{n+1}\right).\]
Finally we apply Theorem \ref{t:th2.5} with the weight $y^{-n}(1+|\log y|^{n+1})$ and $\alpha(\theta)=\theta/4$.
\subsection{The weight lemma}\label{s:lh3} The aim of the rest of this section is to prove Theorem \ref{l:lh3}.
Before proceeding further we need to introduce some additional notation.
Since $\tilde{u}$ is continuous up to the boundary it has some boundary values which we denote by $\tilde{u}(\cdot,0)$. 
Fix any $\beta \in [0,\frac{\pi}{2}]$ and let
\begin{equation}\notag
A = A(0,\beta) = \{x\in S: \phi(x,\eta)\leq \beta\},\; a = S\setminus A.
\end{equation} 
 The main ingredient in the proof of Theorem \ref{l:lh3} is the following lemma
\begin{lemma}\label{l:surfacelemma}
Let $k:\mathbb{R}_+\mapsto\mathbb{R}_+$ be a strictly decreasing absolutely continuous function such that
\begin{subequations}\label{e:kavl}
\begin{eqnarray}
\label{e:kavl.1}&k(0)  \leq \frac{\lambda}{\beta^n};\\
\label{e:kavl.2}&\int_{0}^{1}\left(\frac{k(y)}{y}\right)^{\frac{1}{n+1}}\,dy \leq \lambda^{\frac{1}{n+1}} 
\end{eqnarray}
\end{subequations}
for some positive $\lambda\leq\frac{1}{3\pi}$.
There exist a domain $\Omega\subset B$  and a positive function $v_a$, harmonic in $\Omega$, such that
\begin{subequations}
\begin{eqnarray}
\label{e:sf.1}&A\subset \partial\Omega,\; 0\in \Omega;\\
\label{e:sf.3}&v_a(0) \leq C(n)\lambda^{\frac{1}{n+1}};\\
\label{e:sf.4}&v_a(x,y) \gtrsim k(y) \gtrsim \mu(x,y,\beta),\quad (x,y)\in \partial\Omega\setminus A, 
\end{eqnarray}
\end{subequations}
where the constants depend only on dimension.
\end{lemma}
This lemma uses the modification of the argument presented in Lemma \textbf{4} in \cite{KRS}. Basically it allows us to estimate the average of the weight $k$ on $\partial\Omega\setminus A$ with respect to the harmonic measure of $\Omega$ at zero. The key point here is the second inequality in \eqref{e:sf.4} which will be used later (in \ref{ss:sf.2})   to obtain the lower bound in \eqref{e:tsae}.

\subsection{Proof of Lemma \ref{l:surfacelemma}: auxiliary surface $\Gamma_a$}
To obtain $\Omega$ we construct its boundary $\partial\Omega = \Gamma_a \bigcup A$. The surface $\Gamma_a$ is defined below, the idea is to construct $\Gamma_a$ on which the second inequality in \eqref{e:sf.4} is satisfied, moreover $k(y) \sim \mu(x,y,\beta)$.

Formally, consider the function $\frac{y}{k(y)}$, which is strictly increasing. Let $s = s(\beta)$ be the solution of the following equation on $y$
\begin{equation}\notag
\frac{y}{k(y)} = \beta^{n+1}.
\end{equation}
Since $k$ is decreasing,  \eqref{e:kavl.1} implies
\begin{equation}\label{e:sbdb}
s = k(s)\beta^{n+1} \leq k(0)\beta^{n+1} \leq \lambda\beta.
\end{equation}
Further, \eqref{e:kavl.2} and the monotonicity of $k$ implies that
\begin{equation}\notag
k^{\frac{1}{n+1}}(1)\int_{0}^1\left(\frac{1}{y}\right)^{\frac{1}{n+1}}\,dy \leq \int_{0}^1\left(\frac{k(y)}{y}\right)^{\frac{1}{n+1}}\,dy \leq \lambda^{\frac{1}{n+1}},
\end{equation}
therefore
$\frac{1}{k(1)} \geq \frac{1}{3\lambda}  \geq \pi.$
Now if we let $\rho = \rho(\beta)$ be the solution of
\begin{equation}\notag
\frac{y}{k(y)} = (\pi-\beta)^{n+1},
\end{equation}
we see that $\rho<1$. Further let
\begin{equation}\label{e:gammaadef}
\begin{split}
&\gamma(y) = \beta + \left(\frac{y}{k(y)}\right)^{\frac{1}{n+1}},\quad s\leq y\leq \rho;\\
&\gamma(y) = \beta + \left(\frac{y}{k(y)\beta^{n-1}}\right)^{\frac{1}{2}},\quad 0\leq y \leq s,
\end{split}
\end{equation}
note that $\gamma(\rho) = \pi$.
The surface $\Gamma_a$ is defined as follows
\begin{equation}\notag
\Gamma_a := \{(x,y): \phi(x,\eta) = \gamma(y), \; x\in a = S\setminus A,\; y \in [0,\rho]\},
\end{equation}
and we define $\Omega$ as the domain bounded by $A\bigcup \Gamma_a$, so that $\Omega$
satisfies \eqref{e:sf.1}.
\subsection{Proof of Lemma \ref{l:surfacelemma}: auxiliary function $v_a$}
We define $v_a$ on the boundary of the unit ball by 
\begin{equation}\label{e:vadef}
\begin{split}
&v_a(x,0) = k(y),\quad (x,y)\in \Gamma_a;\\
&v_a(x,0) = 0, \quad x \in A;
\end{split}
\end{equation}
$v_a$ being the harmonic continuation of $v_a(\cdot,0)$ to the ball. Note that the function $v_a$ is axially symmetric. 
It remains to verify \eqref{e:sf.3} and \eqref{e:sf.4}.

In what follows the letter $C$ denotes a constant, depending only on $n$, which value can change from line to line.
The proof of the first inequality is straightforward if somewhat cumbersome. We have $\gamma^{-1}(\beta) = 0,\,\gamma^{-1}(2\beta) = s$ and it follows from \eqref{e:sprf} that 
\begin{multline*}
v_a(0) = C\int_{\beta}^{\pi}\int_{\phi(x,\eta)=t}v_a(x,0)\,d\sigma_{n-1}(x)\,dt =
C\int_{\beta}^{\pi}k(y(\gamma))\sin^{n-1}\gamma\,d\gamma =\\  C\int_{\beta}^{2\beta}k(y(\gamma))\sin^{n-1}\gamma\,d\gamma +
C\int_{2\beta}^{\pi}k(y(\gamma))\sin^{n-1}\gamma\,d\gamma=\\ C\int_{0}^{s}k(y)\sin^{n-1}(\gamma(y))\gamma'(y)\,dy  + C\int_{s}^{\rho}k(y)\sin^{n-1}(\gamma(y))\gamma'(y)\,dy
\end{multline*}
These two integrals are dealt with more or less in the same way. For the first one we have $\gamma(y) \leq \gamma(s) = 2\beta$ so that $\sin\gamma(y) \leq 2\beta$. Then
\begin{multline*}
\int_{0}^{s}k(y)\sin^{n-1}(\gamma(y))\gamma'(y)\,dy \leq C\beta^{n-1}\int_{0}^{s}k(y)\gamma'(y)\,dy =\\
C\beta^{\frac{n-1}{2}}\int_{0}^{s}y^{-\frac12}\left(k^{\frac{1}{2}}(y) - yk'(y)k^{-\frac12}(y)\right)\,dy\leq\\
C\beta^{\frac{n-1}{2}}\left(\int_{0}^{s}\sqrt{\frac{k(y)}{y}}\,dy + \int_{0}^{s}y^{\frac12}\,dk^{\frac{1}{2}}(y)\right)\leq\\
C\beta^{\frac{n-1}{2}}\sqrt{k(0)s} \leq C\beta^{\frac{n-1}{2}}\sqrt{\lambda^2\cdot\beta^{1-n}} \leq C\lambda,
\end{multline*}
the next to last inequality follows from \eqref{e:kav.1} and \eqref{e:sbdb}. 
Analogously, for the second integral we have $y\in [s,\rho]$ and $\gamma'(y) = \frac{1}{n+1}\left(\frac{y}{k(y)}\right)^{-\frac{n}{n+1}}\frac{k(y)-yk'(y)}{k^2(y)}$, also $\gamma(y) \leq 2\left(\frac{y}{k(y)}\right)^{\frac{1}{n+1}}$. We get
\begin{multline*}
\int_{s}^{\rho}k(y)\sin^{n-1}(\gamma(y))\gamma'(y)\,dy \leq \int_{s}^{\rho}k(y)(\gamma(y))^{n-1}\gamma'(y)\,dy \leq\\
2^{n-1}\int_{0}^1k(y)\frac{k^{-\frac{n-1}{n+1}}(y)y^{\frac{n-1}{n+1}}}{n+1}\left(y^{-\frac{n}{n+1}}k^{-\frac{1}{n+1}}(y) - y^{\frac{1}{n+1}}k^{-\frac{1}{n+1}-1}(y)k'(y)\right)\,dy\leq\\
\frac{2^{n-1}}{n+1}\left(\int_0^1\left(\frac{k(y)}{y}\right)^{\frac{1}{n+1}}\,dy - (n+1)\int_0^1\,y^{\frac{n}{n+1}}\,d(k^{\frac{1}{n+1}}(y)) \right)=\\
\frac{2^{n-1}}{n+1}\left(\int_0^1\left(\frac{k(y)}{y}\right)^{\frac{1}{n+1}}\,dy + n\int_0^1\left(\frac{k(y)}{y}\right)^{\frac{1}{n+1}}\,dy - (n+1)k^{\frac{1}{n+1}}(1)\right) \leq \\
C\int_0^1\left(\frac{k(y)}{y}\right)^{\frac{1}{n+1}}\,dy \leq C\lambda^{\frac{1}{n+1}}.
\end{multline*}
Combining these two estimates, we obtain that
$v_a(0) \leq C(\lambda + \lambda^{\frac{1}{n+1}}) \lesssim \lambda^{\frac{1}{n+1}},$
since $\lambda \leq 1$.

The second part of \eqref{e:sf.4}, i.e. the inequality
$
k(y) \gtrsim \mu(x,y,\beta),$ for $(x,y)\in \Gamma_a
$
follows directly from Lemma \ref{l:pki}. 
Indeed, Lemma \ref{l:pki} implies that for $t\leq \phi(x,\eta) \leq \frac{\pi}{2}$
\begin{equation}\label{e:mueqs}
\mu(x,y,t) \sim \frac{y}{((\phi(x,\eta)- t)^2+y^2)\left(((\phi(x,\eta)- t)^2+y^2)^{\frac{n-1}{2}} + \phi^{n-1}(x,\eta)\right)}.
\end{equation}
If $\beta\leq \phi(x,\eta)\leq 2\beta$ then  clearly $((\phi(x,\eta)- \beta)^2+y^2)^{\frac{n-1}{2}} + \phi^{n-1}(x,\eta) \geq \beta^{n-1}$. Further $(\phi(x,\eta)- \beta)^2 = \frac{y}{k(y)\beta^{n-1}}$ for $(x,y)\in \Gamma_a$, therefore we get
\begin{multline*}
\mu(x,y,\beta) \leq \frac{Cy}{((\phi(x,\eta)- \beta)^2+y^2)\beta^{n-1}}\leq \frac{Cy}{(\phi(x,\eta)- \beta)^2\beta^{n-1}}\leq\\
 \frac{Cyk(y)\beta^{n-1}}{y\beta^{n-1}}\leq Ck(y), \quad(x,y)\in \Gamma_a.
\end{multline*}
It follows from \eqref{e:mueqs} that $\mu(x,y,\beta) \leq \frac{cy}{(\phi(x,\eta)-\beta)^{n+1}}$. For $2\beta\leq \phi(x,\eta)\leq \pi$ we have  $4(\phi(x,\eta)- \beta)^2 \geq \phi^2(x,\eta)$, and therefore due to \eqref{e:gammaadef}
\begin{equation}\notag
\mu(x,y,\beta) \leq \frac{Cy}{(\phi(x,\eta)- \beta)^{n+1}} = \frac{Cy}{(\gamma(y)- \beta)^{n+1}} = \frac{Ck(y)y}{y} = Ck(y),\quad (x,y)\in \Gamma_a.
\end{equation}
To obtain the first part of \eqref{e:sf.4}, we first show that 
\begin{equation}\label{e:ylphb}
y \leq \phi(x,\eta)-\beta
\end{equation}
for $(x,y) \in \Gamma_a$. Indeed, for $2\beta \leq \phi(x,\eta) $ it follows from \eqref{e:kavl.2} and \eqref{e:gammaadef} that
\begin{equation}\notag
\frac{y}{\phi(x,\eta)-\beta} = k^{\frac{1}{n+1}}(y)y^{1-\frac{1}{n+1}} \leq \int_0^{y}\left(\frac{k(\tau)}{\tau}\right)^{\frac{1}{n+1}}\,d\tau \leq \lambda^{\frac{1}{n+1}},
\end{equation}
since $\frac{k(y)}{y}$ is decreasing. If $\beta\leq \phi(x,\eta)\leq 2\beta$ then $y\leq s\leq\lambda\beta$ and \eqref{e:gammaadef} gives
\begin{equation}\notag
\frac{y}{\phi(x,\eta)-\beta} = k^{\frac{1}{2}}(y)\beta^{\frac{n-1}{2}}y^{\frac{1}{2}} \leq k^{\frac{1}{2}}(0)\beta^{\frac{n-1}{2}}s^{\frac{1}{2}}\leq (\lambda\beta^{-n})^{\frac{1}{2}}\beta^{\frac{n-1}{2}}(\lambda\beta)^{\frac{1}{2}} \leq\lambda\leq 1.
\end{equation}
Put $E(x,y) = \{\xi\in S: \phi(\xi,\eta)\leq \phi(x,\eta),\; \phi(x,\xi)\leq y\}$. It follows from \eqref{e:ylphb} that $\sigma_n(E(x,y))\sim y^{n}$ for $(x,y)\in \Gamma_a$ and $E(x,y)\subset a$. Since for such a $\xi$ we have $P_y(\xi,x)\gtrsim\frac{1}{y^n}$, the function $v_a(x,0)$ is axially symmetric and by definition strictly decreasing with respect to $\phi(x,\eta)$ for $\phi(x,\eta)\geq \beta$, we get
\begin{multline*}
v_a(x,y) = \int_{S}v_a(\xi,0)P_y(\xi,x)\,d\sigma_n(\xi) \geq\\ \int_{\{\xi\in a: \phi(\xi,\eta)\leq \phi(x,\eta)\}}v_a(\xi,0)P_y(\xi,x)\,d\sigma_n(\xi)\gtrsim
\int_{E(x,y)}v_a(\xi,0)\frac{1}{y^n}\,d\sigma_n(\xi)\geq\\ \frac{1}{y^n}\int_{E(x,y)}v_a(x,0)\,d\sigma_n(\xi)
\geq \frac{1}{y^n}v_a(x,0) \sigma_n(E(x,y))\gtrsim Ck(y).
\end{multline*}
This completes the proof of Lemma \ref{l:surfacelemma}.
\subsection{Proof of Theorem \ref{l:lh3}}\label{ss:sf.2}
First we renormalize the weight $\tilde{k}$ and the function $\tilde{u}$, let
\begin{equation}\label{e:lambdan}
\begin{split}
&k(y) = \frac{\lambda}{D^{n+1} + \tilde{k}(0)\beta^n} \tilde{k}(y);\\
&u(z) = \frac{\lambda}{D^{n+1} + \tilde{k}(0)\beta^n} \tilde{u}(z),
\end{split}
\end{equation}
where $\lambda = \lambda(n) \leq \frac{1}{3\pi}$ is a small positive constant to be chosen later. 
We may assume that $x_0 = \eta$ and that $\tilde{u}(\cdot,0)$ (and therefore $u(\cdot,0)$) is a axially symmetric function, 
\begin{equation}\notag
u(x,0) = \varphi(|\phi(x,\eta)|),\quad x \in S.
\end{equation}
By $u_A$ and $u_a$ we denote the harmonic continuation to $B$ of the functions $u(\cdot,0)\cdot\chi_A$ and $u(\cdot,0)\cdot\chi_a$ correspondingly.

Clearly $0=u(0) = u_a(0) + u_A(0)$.
Let
\begin{equation}\notag
K = -u_A(0) = -\int_{A}u(x,0)\,d\sigma_n(x),
\end{equation}
we may assume that $K\geq 0$ (otherwise \eqref{e:tsae} is trivial). We see that \eqref{e:kav.1} and \eqref{e:kav.2} imply that the weight $k$ satisfies the conditions \eqref{e:kavl.1} and \eqref{e:kavl.2}. Let $\Gamma_a$ and $v_a$ be like in Lemma \ref{l:surfacelemma}. Our first aim is to prove the following inequality
\begin{equation}\label{e:sf.2}
u_a(x,y) \leq C(1+K)v_a(x,y),\quad (x,y)\in \Gamma_a.
\end{equation}
Since $u_a(\cdot,0)$ is just the part of the boundary values of $u$ that lies in $a$, we have
\begin{equation}\notag
u_a(x,y) = u(x,y) - u_A(x,y) \leq k(y) - u_A(x,y),\quad (x,y) \in \Gamma_a ,
\end{equation}
so to get an upper estimate on $u_a$ we actually need to bound $u_A$ from below on $\Gamma_a$. Again, \eqref{e:sprf} provides us 
with 
\begin{equation}\notag
u_A(x,y)= C(n) \int_{0}^{\beta}\varphi(t)\mu(x,y,t)\sin^{n-1} t\,dt,
\end{equation}
in particular we have
\begin{equation}\label{e:uastr}
u_A(0) = C(n) \int_{0}^{\beta}\varphi(t)\sin^{n-1} t\,dt.
\end{equation}
Clearly $\varphi(t) - k(0) \leq0$, therefore using the mean value theorem (the first one, unlike in \cite{KRS}) we see that there exists $t_0\in [0,\beta]$ such that
\begin{multline*}
\int_{0}^{\beta}\varphi(t)\mu(x,y,t)\sin^{n-1} t\,dt =\\ \int_{0}^{\beta}\left(\varphi(t) - k(0)\right)\mu(x,y,t)\sin^{n-1} t\,dt +  \int_{0}^{\beta}k(0)\mu(x,y,t)\sin^{n-1} t\,dt = \\
\mu(x,y,t_0)\int_{0}^{\beta}\left(\varphi(t) - k(0)\right)\sin^{n-1} t\,dt + k(0)\int_{0}^{\beta}\mu(x,y,t)\sin^{n-1} t\,dt \geq\\
\mu(x,y,t_0)\int_{0}^{\beta}\varphi(t)\sin^{n-1} t\,dt -\mu(x,y,t_0)k(0)\int_{0}^{\beta}\sin^{n-1} t\,dt =\\ \mu(x,y,t_0)\frac{u_A(0)}{C(n)} -\mu(x,y,t_0)k(0)\int_{0}^{\beta}\sin^{n-1} t\,dt,
\end{multline*}
the last equality follow from \eqref{e:uastr}. Now \eqref{e:lambdan} implies that $k(0)\beta^n  \leq \lambda \leq 1$, we also have $\int_{0}^{\beta}\sin^{n-1} t\,dt\sim \beta^{n}$. We continue the estimate, obtaining 
\begin{multline*}
\int_{0}^{\beta}\varphi(t)\mu(x,y,t)\sin^{n-1} t\,dt \geq\mu(x,y,t_0)\left(\frac{u_A(0)}{C(n)} - k(0)\int_{0}^{\beta}\sin^{n-1}t\,dt\right) \geq \\
\mu(x,y,t_0)\left(-\frac{K}{C(n)} - C(\beta)\lambda\right), 
\end{multline*}
where $C(\beta) \sim 1$.
It follows from \eqref{e:mueqs} that 
$\sup_{0\leq t\leq\beta}\mu(x,y,t) \sim \mu(x,y,\beta),$ when $\phi(x,\eta)>\beta.$
We therefore have
\begin{multline*}
u_A(x,y) \geq  C(n)\mu(x,y,t_0)\left(-\frac{K}{C(n)} - C(\beta)\lambda\right) \geq\\
C(n)\mu(x,y,\beta)\left(-\frac{K}{C(n)} - C(\beta)\lambda\right) \geq - \mu(x,y,\beta)\cdot (K+C) \geq \\
-C(n)(K+1)\mu(x,y,\beta).
\end{multline*}
Gathering all the estimates and applying \eqref{e:sf.4}, we get
\begin{equation}\notag
u_a(x,y) \leq k(y) - u_A(x,y) \leq k(y) + C(n)(K+1)\mu(x,y,\beta) \lesssim v_a(x,y)(K+1)
\end{equation}
for $(x,y)\in \Gamma_a$, and we obtain \eqref{e:sf.2}.

Once we have this estimate it is quite easy to finish the proof. Indeed, it follows from \eqref{e:sf.2},\eqref{e:sf.3} and the maximum principle that
\begin{equation}\notag
K = u_a(0) \leq C(1+K)v_a(0) \leq C_0\lambda^{\frac{1}{n+1}}(1+K) \leq \frac{1+K}{3}
\end{equation}
for sufficiently small $\lambda$. Therefore we have
$K \leq \frac12,$ which means that 
\begin{equation}\notag
\int_{A(0,\beta)}\tilde{u}(x,0)\,d\sigma_n(x)  \geq -\frac{1}{2\lambda}\left(D ^{n+1}+\tilde{k}(0)\beta^{n}\right) ,
\end{equation}
and we are done.
\section{Concluding remarks}\label{s:ctrx}
  \subsection{Borichev's conditions}\label{s:bc}
In this subsection we compare our regularity restrictions on the weight with those given in \cite{Bor:03} for two-dimensional results. To do this we quote two following theorems
\begin{theoremmm}\label{t:tb1}
Let $u$ be a harmonic function in the unit disc such that $u(0)=0$ and
$u(z) \leq w(1-|z|),$ when $z\in \mathbb{D}.$
Put $\psi(t) = \log w(e^{-t})$ and assume that
\begin{equation}\label{e:bc.1}
\begin{split}
&1< 1+\varepsilon = \lim \psi'(t) < +\infty;\\
&\psi' \; \text{is of bounded variation on } \mathbb{R}.
\end{split}
\end{equation}
Then
\begin{equation}\label{e:ta}
u(z) \geq -C-\left(o(1)+\cos\left(\frac{\pi}{2+\varepsilon}\right)^{-2-\varepsilon}\right)w(1-|z|),\quad z\in \mathbb{D},
\end{equation}
where $C$ depends only on $w$.
\end{theoremmm}
\begin{theoremmm}\label{t:tb2}
Let $u$ be a harmonic function in the unit disc such that $u(0)=0$ and
$u(z) \leq w(1-|z|),$ when $z\in \mathbb{D}$. Put $\psi(t) = \log w(e^{-t})$ and assume that
\begin{subequations}\label{e:bc.2}
\begin{eqnarray}
\label{e:bc.2.1}&\lim \psi'(t) = +\infty;\\
\label{e:bc.2.2}&|\psi''(t)|  = O\left(|\psi'|^{2-\tilde{\varepsilon}}(t)\right),\quad t\rightarrow+\infty.
\end{eqnarray}
\end{subequations}
for some  positive $\tilde{\varepsilon}$.
Then
\begin{equation}\label{e:tb}
u(z) \geq - C -(1+o(1))w(1-|z|),\quad z\in \mathbb{D}.
\end{equation}
where $C$ depends only on $w$.
\end{theoremmm}
As we see, Theorems \ref{t:tb1} and \ref{t:tb2} provide better growth estimates than Theorem \ref{t:th1}, in particular the constant $C$ in \eqref{e:utsgc} is replaced by $(1+o(1))$ in \eqref{e:tb}. We will modify the conditions \eqref{e:bc.1} for the $(n+1)$-dimensional setting and show that they imply \eqref{e:utsgc} (the constant though will not be as nice as in \eqref{e:ta}). Analogously we show that \eqref{e:utsgc} follows from the conditions \eqref{e:bc.2.1} and \eqref{e:bc.2.2} as well, moreover they actually imply regularity conditions \eqref{e:als} and \eqref{e:ar}.
\subsection{Modifying Theorems \ref{t:tb1} and \ref{t:tb2}: regularity conditions in higher dimensions.}
The hypothesis of Theorem \ref{t:tb1} is roughly speaking that $w$ is a regular weight of polynomial growth. The multidimensional version of \eqref{e:bc.1} looks as follows
\begin{equation}\label{e:bc.1.n}
\begin{split}
&n< n+\varepsilon = \lim_{t\rightarrow\infty} \psi'(t)  < +\infty;\\
&\psi' \; \text{is of bounded variation on } \mathbb{R}
\end{split}
\end{equation}
for some $\varepsilon > 0$.
This condition \eqref{e:bc.1.n}, strictly speaking, does not imply \eqref{e:als} and \eqref{e:ar} (in our regularity conditions we are asking for $w$ to be a little bit smoother than in \eqref{e:bc.1}), nevertheless we can still estimate $U$ from below.  
Indeed, put $N = \sup_{t\in\mathbb{R}}\psi'(t)$. We see that
\[
\psi'(t) = -\frac{w'(e^{-t})e^{-t}}{w(e^{-t})},\quad t\in\mathbb{R},
\]
therefore it follows from \eqref{e:bc.1.n} that $-\frac{N}{y} \leq \frac{w'(y)}{w(y)} \leq -\frac{n+\varepsilon}{y},\; 0<y\leq y_0\leq 1$. This situation was already discussed in the section \ref{ss:sc}, and we get \eqref{e:wnr}. 

The multidimensional versions of the regularity and growth conditions in Theorem \ref{t:tb2} are literally the same. We show that \eqref{e:bc.2.1} and \eqref{e:bc.2.2} imply \eqref{e:als} and \eqref{e:ar} near the boundary.
Let us rewrite the conditions \eqref{e:bc.2} in terms of the weight $w$. Since
\[
\begin{split}
&\psi''(t) = \left(\log w(e^{-t})\right)'' = \frac{w''(e^{-t})e^{-2t}}{w(e^{-t})} + \frac{w'(e^{-t})e^{-t}}{w(e^{-t})} - \frac{(w')^2(e^{-t})e^{-2t}}{w^2(e^{-t})} =\\
&\frac{w''(e^{-t})e^{-2t}}{w(e^{-t})} - \psi'(t) + (\psi')^2(t)  = \frac{w''(y)y^2}{w(y)} + \frac{w'(y)y}{w(y)} - \frac{(w')^2(y)y^2}{w^2(y)},\quad e^{-t} = y,
\end{split}
\]
we see that \eqref{e:bc.2.2} is equivalent to 
\begin{multline*}
\limsup_{t\rightarrow\infty}\frac{|\psi''(t)|}{|\psi'|^{2-\tilde{\varepsilon}}(t)} = \lim_{t\rightarrow\infty}\left|\frac{w''(e^{-t})e^{-2t}}{w(e^{-t})} + \psi'(t) - (\psi')^2(t)\right|\cdot|\psi'|^{\tilde{\varepsilon}-2}(t) =\\
\limsup_{t\rightarrow\infty}\left|\frac{w''(e^{-t})e^{-2t}e^{-t(\tilde{\varepsilon}-2)}|w'|^{\tilde{\varepsilon}-2}(e^{-t})}{w^{\tilde{\varepsilon}-2}(e^{-t})w(e^{-t})} + |\psi'|^{\tilde{\varepsilon}-1}(t) - |\psi'|^{\tilde{\varepsilon}}(t)\right| =\\
\limsup_{t\rightarrow\infty}\left|\frac{w''(e^{-t})w(e^{-t})}{w'^{2}(e^{-t})}\cdot\left|\frac{e^{-t}w'(e^{-t})}{w(e^{-t})}\right|^{\tilde{\varepsilon}} + |\psi'|^{\tilde{\varepsilon}-1}(t) - |\psi'|^{\tilde{\varepsilon}}(t)\right| =\\
\limsup_{t\rightarrow\infty}|\psi'|^{\tilde{\varepsilon}}(t)\cdot\left|\frac{w''(e^{-t})w(e^{-t})}{w'^{2}(e^{-t})}-1 + |\psi'|^{-1}(t)\right|\leq C < \infty
\end{multline*}
for some positive $\tilde{\varepsilon}$. Combining this with \eqref{e:bc.2.1} and changing variables we obtain
\begin{equation}\label{e:tmp}
\left|\frac{w''(y)w(y)}{w'^{2}(y)}-1 + o(1)\right| = o\left(\frac{w(y)}{yw'(y)}\right)^{\tilde{\varepsilon}_1},\quad y\rightarrow 0,
\end{equation}
for all positive $\tilde{\varepsilon}_1 \leq \tilde{\varepsilon}$ and
\[
\lim_{y\rightarrow0}\frac{w(y)}{yw'(y)} = 0,
\]
so \eqref{e:als} holds. If we look now at \eqref{e:ar}, we see that it can be rewritten as follows
\begin{equation}\label{e:tmp1}
\frac{w''(y)w(y)}{(w')^2(y)} - 1 \leq \frac{1-\delta}{n}, \quad 0<y\leq 1,
\end{equation}
for some positive $\delta$. Again, \eqref{e:tmp} does not necessarily imply \eqref{e:tmp1} for \textit{all} values of $y\in(0,1]$, but as long as we have \eqref{e:tmp1} for $0<y\leq y_0$ for \textit{some} $y_0$ (see section \ref{ss:sc}), we can still get \eqref{e:wnr}. It follows that the conditions \eqref{e:ar} and \eqref{e:als} are more general than \eqref{e:bc.2}.
\subsection{An example}\label{ss:ex}
The aim of this section is to prove that the estimate in Theorem \ref{th:N3} is the best possible up to a constant. We will construct a harmonic function in the unit ball that satisfies
\[U(z)\le C(1-|z|)^{-n}\quad{\text{and}}\quad U(r)\le -C_1(1-r)^{-n}\left(\log\frac1{1-r}\right)^{n+1}.\]
For $n=1$ the example is relatively simple, one may take \[U(z)=\Re(-(1-z)(\log(1-z))^2),\] see the original work of M. Cartwright \cite{Car:33}. For $n\ge 2$ we will construct an axially symmetric harmonic function, see \cite{W}. 
We are looking for a solution of the following equation
\[
\frac{\partial^2 V}{\partial x^2}+\frac{\partial^2 V}{\partial y^2}+\frac{n-1}{y}\frac{\partial V}{\partial y}=0,\ y>0.\] 
If we do the change of variables $z'=x'+iy'=\rho e^{i\phi}=1-z$ we see that it suffices to construct a function in the domain  $x'>0, y'\ge 0, (x'-1)^2+y'^2\le 1$, that satisfies the equation above. Further, writing it down in polar coordinates, we obtain
\[
\frac{n}{\rho}\frac{\partial V}{\partial \rho}+\frac{\partial^2 V}{\partial \rho^2}+\frac{n-1}{\rho^2}\frac{\partial V}{\partial \phi}\cot \phi+\frac{1}{\rho^2}\frac{\partial^2 V}{\partial \phi^2}=0.\]
We also want $V(t,0)\le -C_1t^{-n}\log^{n+1}\frac1{t}$ and $V(\rho,\phi)\le C\rho^{-n}\cos\phi^{-n}$.  Let us look for $V$ in the following form
\[V(\rho,\phi)=\rho^{-n}\sum_{k=0}^{n+1} v_k(\phi)\left(\log\frac1{\rho}\right)^k.\]
Then
\[
\frac{\partial V(\rho, \phi)}{\partial \rho}=-n\rho^{-1}V(\rho,\phi)-\rho^{-n-1}\sum_{k=0}^n k v_k(\phi)\left(\log\frac1{\rho}\right)^{k-1},\]
\begin{multline*}
\frac{\partial^2 V(\rho, \phi)}{\partial \rho^2}=n(n+1)\rho^{-2}V(\rho, \phi)+\\
\rho^{-n-2}\sum_{k=0}^nkv_k(\phi)\left((2n+1)\left(\log\frac{1}{\rho}\right)^{k-1}+
(k-1)\left(\log\frac1{\rho}\right)^{k-2}\right).
\end{multline*}
This gives the system of equations for $v_k$, $k=0,...,n+1,$
\[
v_k''+(n-1)v_k'\cot\phi+nv_k+(n+1)(k+1)v_{k+1}+(k+2)(k+1)v_{k+2}=0,\]
where $v_{n+2}=v_{n+3}=0$. Applying one more change of variable $v_k(\phi)=f_k(\cos\phi)$, we get
the following system of equations on $f_j(t)$
\begin{equation}\label{eq:sf}
(1-t^2)f_k''-ntf_k'+nf_k+(k+1)((n+1)f_{k+1}+(k+2)f_{k+2})=0.\end{equation}
We will need the auxiliary result that might be standard for the specialists.  
\begin{lemma}\label{l:l5} Let $n\ge 2$ be an integer. For any $r\in C([0,1])$ there exists $f\in C([0,1])$ that solves
\[(1-t^2)f''-ntf'+nf=r(t).\]
\end{lemma}
\begin{proof} The equation we consider has a regular singular point at $t=1$.
First, we look at two linearly independent solutions of the homogeneous equation. We have $f_1(t)=t$, and, by the Frobenius method (see for example \cite[Chapter 4]{T}),
\[
f_2(t)=\begin{cases} 
t\log(1-t)+ a(t),\ n=2,\\
(1-t)^{1-\frac{n}{2}}b(t),\ n\ {\text{is odd}},\ n\ge 3,\\
At\log(1-t)+(1-t)^{1-\frac{n}{2}}c(t),\ n\ {\text{is even}},\ n\ge 4,
\end{cases}\]
where $a,b,c$ are analytic near $[0,1]$ and $b(1)\neq 0,\ c(1)\neq 0$. Let further $W(t)=f_1f_2'-f_1'f_2$ be the Wronskian of $f_1$ and $f_2$. Then
\[
W(t)=
(1-t)^{-n/2}g(t)
\]
where $g$ is analytic near $[0,1]$ and does not vanish on $[0,1]$ (a simple computation shows that it does not vanish at $1$ since $b$ and $c$ do not vanish, and on $[-\delta,1)$ the functions $f_1, f_2$ are two linearly independent solutions of a second-order linear differential equation with regular coefficients).
The bounded solution that we look for is given by the variation of parameters formula
\[
f(t)=-t\int_0^t r(s)g^{-1}(s)(1-s)^{n/2}f_2(s)ds-f_2(t)\int_t^1 sr(s)g^{-1}(s)(1-s)^{n/2}ds.\]
The first term is continuous on $[0,1]$  since $(1-s)^{n/2}f_2(s)\in C([0,1])$, in the second term  the integral satisfies $I(t)=O((1-t)^{n/2+1}),\ t\rightarrow 1$ and $f_2(t)I(t)\in C([0,1])$. 
\end{proof}

We start by choosing $f_{n+1}(t)=-t$. Then by Lemma there exist bounded functions $f_1,..., f_{n}$ on $[0,1]$ such that (\ref{eq:sf}) holds.  We get
\[\sum_{k=0}^{n}v_k(\phi)\log^{k}(1/\rho)\le M\log^n(1/\rho).\] 
Now if we assume that $V(\rho,\phi)>0$ then $\cos\phi\log(1/\rho)<M$ and 
\begin{multline*}
-\cos\phi\left(\log\frac1{\rho}\right)^{n+1}+\sum_{k=0}^nv_k(\cos\phi)\left(\log\frac1{\rho}\right)^k\le\\
-\cos\phi\left(\log\frac1{\rho}\right)^{n+1}+M\left(\log\frac1{\rho}\right)^n\le
M^{n+1}(\cos\phi)^{-n}.
\end{multline*}
Therefore we obtain the desired estimate $V(\phi,\rho)\le C\rho^{-n}\cos\phi^{-n}$.

\section*{Acknowledgements}
The work was started in the Center of Advanced Study at the Norwegian Academy of Science and Letters in Oslo, and finished in Norwegian University of Science and Technology. We are grateful to both institutions.\par
We would like to thank Philip Rippon for drawing our attention to the article \cite{R:82}. Thanks also go to Mats Ehrnstr\"{o}m for a useful discussion on Lemma \ref{l:l5} with one of the authors.

\end{document}